\documentclass{amsart}

\usepackage[margin=1.5in]{geometry}

\usepackage{graphicx}
\usepackage[centertags]{amsmath}
\usepackage{amsfonts}
\usepackage{amssymb}
\usepackage{amsthm}
\usepackage{newlfont}
\usepackage{hyperref}

\newtheorem{thm}{Theorem}[section]
\newtheorem{cor}[thm]{Corollary}
\newtheorem{lem}[thm]{Lemma}
\newtheorem{klem}[thm]{Key Lemma}
\newtheorem*{klem*}{Key Lemma}
\newtheorem{prop}[thm]{Proposition}

\theoremstyle{definition}

\theoremstyle{remark}

\numberwithin{equation}{section}

\newcommand{\fact}{f(G)}
\newcommand{\von}{v(G)}
\newcommand{\pmp}{p{$.$}m{$.$}p{$.$}}
\newcommand{\N}{\mathbb{N}}
\newcommand{\Z}{\mathbb{Z}}

\newcommand{\F}{\mathrm{F}}
\newcommand{\acts}{\curvearrowright}
\newcommand{\Sym}{\mathrm{Sym}}

\newcommand{\Aff}{\mathrm{Aff}}

\newcommand{\id}{\mathrm{id}}
\newcommand{\p}{\mathcal{P}}
\newcommand{\sH}{\mathrm{H}}

\renewcommand{\:}{\,:\,}

\begin{document}

\title{Every action of a non-amenable group is the factor of a small action}

\author{Brandon Seward}
\address{Department of Mathematics, University of Michigan, 530 Church Street, Ann Arbor, MI 48109, U.S.A.}
\email{b.m.seward@gmail.com}
\keywords{non-amenable, factor, generating partition, entropy, Polish action}
\subjclass{37A15; 37A20, 37A35, 37B10, 37B40, 03E15}
\thanks{This research was supported by the National Science Foundation Graduate Student Research Fellowship under Grant No. DGE 0718128. The author would like to thank Lewis Bowen, Damien Gaboriau, Alexander Kechris, and Ralf Spatzier for helpful conversations. The author also thanks the referee for many helpful suggestions.}

\begin{abstract}
It is well known that if $G$ is a countable amenable group and $G \acts (Y, \nu)$ factors onto $G \acts (X, \mu)$, then the entropy of the first action must be greater than or equal to the entropy of the second action. In particular, if $G \acts (X, \mu)$ has infinite entropy, then the action $G \acts (Y, \nu)$ does not admit any finite generating partition. On the other hand, we prove that if $G$ is a countable non-amenable group then there exists a finite integer $n$ with the following property: for every probability-measure-preserving action $G \acts (X, \mu)$ there is a $G$-invariant probability measure $\nu$ on $n^G$ such that $G \acts (n^G, \nu)$ factors onto $G \acts (X, \mu)$. For many non-amenable groups, $n$ can be chosen to be $4$ or smaller. We also obtain a similar result with respect to continuous actions on compact spaces and continuous factor maps.
\end{abstract}
\maketitle

\section{Introduction}

Let $G$ be a countably infinite group. For a compact metrizable space $K$, the \emph{(symbolic) Bernoulli shift} $K^G$ is the set of functions $x : G \rightarrow K$, endowed with the topology of point-wise convergence, together with the left coordinate-permutation action of $G$: for $g, h \in G$ and $x \in K^G$, $(g \cdot x)(h) = x(g^{-1} h)$. For a Borel probability measure $\kappa$ on $K$, the probability space $(K^G, \kappa^G)$ is called a \emph{Bernoulli shift}. The space $K^G$ is compact and metrizable, and the action of $G$ on $(K^G, \kappa^G)$ is continuous and measure-preserving. For $n \in \N$ we write $n^G$ for $\{1, 2, \ldots, n\}^G$. 

We assume that all actions of $G$ are Borel actions on standard Borel spaces. Given two actions $G \acts Y$ and $G \acts X$, we say that $G \acts Y$ Borel factors onto $G \acts X$ if there is a $G$-equivariant Borel surjection $\phi: Y \rightarrow X$. When $X$ and $Y$ are topological spaces we furthermore say that $G \acts Y$ continuously factors onto $G \acts X$ if $\phi: Y \rightarrow X$ can be chosen to be continuous. If $\nu$ and $\mu$ are Borel measures on $Y$ and $X$, respectively, then we say that $G \acts (Y, \nu)$ factors onto $G \acts (X, \mu)$ if there is a $G$-equivariant $\nu$-almost-everywhere defined Borel map $\phi: Y \rightarrow X$ such that $\nu$ pushes forward to $\mu$.


Let $G \acts (Y, \nu)$ be a probability-measure-preserving (\pmp) action of $G$ on a standard probability space $(Y, \nu)$, i.e. $Y$ is a standard Borel space and $\nu$ is a $G$-invariant Borel probability measure. A countable Borel partition $\p$ is said to be \emph{generating} if the smallest $G$-invariant $\sigma$-algebra containing $\p$ coincides, up to $\nu$-null sets, with the collection of Borel subsets of $Y$. It is well known that $(Y, \nu)$ admits a $n$-piece generating partition if and only if $G \acts (Y, \nu)$ is isomorphic to $G \acts (n^G, \nu')$ for some invariant Borel probability measure $\nu'$ on $n^G$. If $G$ is a countable amenable group, then a well known property of entropy states that if $G \acts (Y, \nu)$ factors onto $G \acts (X, \mu)$, then $G \acts (Y, \nu)$ is larger than $G \acts (X, \mu)$ in the sense of entropy, meaning that $h_G(Y, \nu) \geq h_G(X, \mu)$. In particular, since the entropy $h_G(Y, \nu)$ is bounded above by $\log |\p|$ for any finite generating partition $\p$, we see that if $h_G(X, \mu) > \log(n)$ then the action $G \acts (Y, \nu)$ does not admit any generating partition having $n$ pieces.

In 1987, Ornstein and Weiss \cite{OW87} discovered the seemingly bizarre property that for the rank two free group $\F_2$, the Bernoulli shift $\F_2 \acts (2^{\F_2}, u_2^{\F_2})$ factors onto the larger Bernoulli shift $\F_2 \acts (4^{\F_2}, u_4^{\F_2})$. Here, for a natural number $n$ we let $u_n$ denote the uniform probability measure on $\{1, 2, \ldots, n\}$. Ball \cite{Ba05} greatly expanded upon this example by proving that for every countable non-amenable group $G$ there is $n \in \N$ so that $G \acts (n^G, u_n^G)$ factors onto $G \acts ([0, 1]^G, \lambda^G)$, where $\lambda$ is Lebesgue measure. So for every non-amenable group we have a specific example of a ``small'' action factoring onto a ``large'' action. Furthermore, Bowen \cite{Bo11} improved upon the Ornstein--Weiss example to prove that if $G$ contains $\F_2$ as a subgroup then in fact all Bernoulli shifts over $G$ factor onto one-another. The main purpose of this paper is to show that such examples are not merely isolated, but rather quite common for actions of non-amenable groups.

\begin{thm} \label{INTRO NA MEAS}
For every countable non-amenable group $G$ there exists a finite integer $n$ with the following property: If $G \acts (X, \mu)$ is any {\pmp} action then there is a $G$-invariant Borel probability measure $\nu$ on $n^G$ such that $G \acts (n^G, \nu)$ factors onto $G \acts (X, \mu)$. If $G$ is not finitely generated, or $G$ has a non-amenable subgroup of infinite index, or $G$ has a subgroup with cost greater than one, then $n$ can be chosen to be $4$. If $G$ contains $\F_2$ as a subgroup then $n$ can be chosen to be $2$.
\end{thm}

We refer the reader to \cite{G00} for the definition of cost. We also obtain the following topological analogue. To contrast the theorem below, we point out that a well known consequence of topological entropy is that if $G$ is a countable amenable group and $n < m$, then there is no compact invariant set $Y \subseteq n^G$ which continuously and equivariantly factors onto $m^G$.

\begin{thm} \label{INTRO NA TOP}
For every countable non-amenable group $G$ there exists a finite integer $n$ with the following property: If $G \acts X$ is any continuous action on a compact metrizable space, then there is a $G$-invariant compact set $Y \subseteq n^G$ such that $G \acts Y$ continuously factors onto $G \acts X$. If $G$ contains $\F_2$ as a subgroup then $n$ can be chosen to be $2$.
\end{thm}

It is an immediate consequence of the above theorems and the definition of sofic entropy (see \cite{KL13}) that for any non-amenable sofic group $G$ there is an integer $n$ so that every action of $G$ is the factor of an action having sofic entropy at most $\log(n)$. This is true both in the setting of {\pmp} actions and measure entropy and in the setting of continuous actions on compact metrizable spaces and topological entropy. Similarly, it is an immediate consequence of the definition of f-invariant entropy (see \cite{Bo10a}) that every {\pmp} action of a non-cyclic finite rank free group is the factor of a {\pmp} action having f-invariant entropy at most $\log(2)$. Finally, since the existence of a factor map implies weak containment (see \cite{K12}) we deduce that, with $G$ and $n$ as in Theorem \ref{INTRO NA MEAS}, every {\pmp} action of $G$ is weakly contained in a {\pmp} action of the form $G \acts (n^G, \nu)$.

The mechanics of the proofs of the above theorems yield the following interesting results which hold for all countably infinite groups. In particular, the results below hold for infinite amenable groups. Recall that a measure $\mu$ on $X$ is \emph{quasi-invariant} for $G \acts X$ if translates of $\mu$-null sets are $\mu$-null. Quasi-invariant measures need not be finite or $\sigma$-finite.

\begin{thm} \label{INTRO MEAS}
Let $G$ be a countably infinite group. For every Borel action $G \acts X$ and every quasi-invariant Borel measure $\mu$ on $X$, there is a quasi-invariant Borel measure $\nu$ on $4^G$ such that $G \acts (4^G, \nu)$ factors onto $G \acts (X, \mu)$.
\end{thm}

Recall that a topological space is \emph{Polish} if it contains a countable dense subset and admits a compatible complete metric. If $X$ is Polish and $Y \subseteq X$, then $Y$ is a \emph{Polish subspace} if the subspace topology on $Y$ is Polish. This is equivalent to $Y$ being a $G_\delta$ subset of $X$ \cite[Theorem 3.11]{K95}.

\begin{thm} \label{INTRO TOP}
Let $G$ be a countably infinite group. For every continuous action $G \acts X$ on a Polish space $X$, there is a $G$-invariant Polish subspace $Y \subseteq 4^G$ such that $G \acts Y$ continuously factors onto $G \acts X$.
\end{thm}

For the sake of completeness we deduce a Borel dynamics version of these results. Its proof is based on a simple application of the Kuratowski--Mycielski theorem \cite[Theorem 19.1]{K95}.

\begin{prop} \label{INTRO BOREL}
Let $G$ be a countable group, and let $G \acts Y$ be a Borel action. Suppose there is an uncountable invariant set $Y_0 \subseteq Y$ such that the action $G \acts Y_0$ is free. Then for every Borel action $G \acts X$ there is a $G$-invariant Borel set $Z \subseteq Y$ such that $G \acts Z$ Borel factors onto $G \acts X$. Moreover, if $X$ has a fixed point then $G \acts Y$ Borel factors onto $G \acts X$.
\end{prop}

\begin{cor} \label{INTRO COR}
Let $G$ be a countably infinite group. Then all (symbolic) Bernoulli shifts over $G$ Borel factor onto one another.
\end{cor}

This corollary is immediate since every Bernoulli shift has a fixed point. We mention that this result is entirely distinct from the question as to which Bernoulli shifts factor onto one another (with their Bernoulli measures) since, for example, for amenable groups entropy implies that $G \acts (2^G, u_2^G)$ does not factor onto $G \acts (4^G, u_4^G)$.

Recall that a set $T \subseteq X$ is a transversal for an action $G \acts X$ if $T$ meets every $G$-orbit precisely once. We mention that Weiss \cite{We82} proved that if $\Z \acts Y$ is a free Borel action which does not admit any Borel transversal, and $\Z \acts X$ is any other Borel action, then there is an invariant Borel set $Z \subseteq Y$ such that the action $\Z \acts Z$ does not admit any Borel transversal and there is an equivariant Borel map $f : Z \rightarrow X$.

All of the above theorems are based on a single key lemma. This lemma roughly states that if $G$ and the rank two free group $\F_2$ both act on a standard Borel space $Z$ with each $\F_2$-orbit contained in a $G$-orbit, and with a sufficiently nice cocycle $\alpha$ relating the actions, then $G \acts 2^G \times Z$ is Borel isomorphic to a nice extension of $G \acts (2^\N)^G \times Z$. The proof of this makes use of the original Ornstein--Weiss example \cite{OW87}. We then deduce each of the above theorems by choosing a suitable space $Z$ with suitable actions of $G$ and $\F_2$. In particular, the proof of Theorem \ref{INTRO NA MEAS} relies on the Gaboriau--Lyons solution \cite{GL09} to the measurable von Neumann conjecture, and the proof of Theorem \ref{INTRO NA TOP} relies on Whyte's solution \cite{Wh99} to the geometric von Neumann conjecture. These techniques furthermore allow us to present a short proof of an unpublished result due to Lewis Bowen. We thank Bowen for his permission to include it here. 

Recall that the \emph{Shannon entropy} $\sH(M, \mu)$ of a probability space $(M, \mu)$ is
$$\sH(M, \mu) = \sum_{m \in M} - \mu(m) \log \mu(m)$$
if there is a countable subset of $M$ having full measure, and $\sH(M, \mu) = \infty$ otherwise.

\begin{thm}[Bowen]
Let $G$ be a countable non-amenable group. Let $\fact$ be the infimum of $\sH(M, \mu)$ over all probability spaces $(M, \mu)$ with the property that $G \acts (M^G, \mu^G)$ factors onto all other Bernoulli shifts over $G$. Let $\von$ be the infimum of $\sH(M, \mu)$ over all probability spaces $(M, \mu)$ with the property that there is an essentially free action of $\F_2$ on $(M^G, \mu^G)$ such that $\mu^G$-almost-every $\F_2$-orbit is contained in a $G$-orbit. Then $\fact = \von < \infty$.
\end{thm}

We remark that, by Bowen's isomorphism theorem \cite{Bo12}, if a probability space $(M, \mu)$ contains at least three points in its support and satisfies $\sH(M, \mu) > \fact = \von$ then $G \acts (M^G, \mu^G)$ factors onto all other Bernoulli shifts over $G$, and there is an essentially free action of $\F_2$ on $(M^G, \mu^G)$ such that $\mu^G$-almost-every $\F_2$-orbit is contained in a $G$-orbit.

The quantity $\fact = \von$ is known to be finite provided $G$ is non-amenable. The finiteness of $\fact$ is due to Ball \cite{Ba05}, and the finiteness of $\von$ is due to Gaboriau and Lyons \cite{GL09}. Bowen \cite{Bo11} proved that $\fact = 0$ if $G$ contains $\F_2$ as a subgroup, and Gaboriau--Lyons \cite{GL09} proved that $\von = 0$ if $G$ has a subgroup of cost greater than one. We observe in Lemma \ref{LEM GL} that $\fact = \von = 0$ if the non-amenable group $G$ is either not finitely generated or has a non-amenable subgroup of infinite index. It is unknown if $\fact = \von = 0$ for all countable non-amenable groups. The quantity $\fact = \von$ has significance to our main theorem, Theorem \ref{INTRO NA MEAS}. Specifically, in that theorem one can always use $n = 2 m$ where $m \geq 3$ satisfies $\fact = \von < \log(m)$. 

In Section \ref{SEC KEY LEM} below we prove the key lemma, in Section \ref{SEC MEAS THMS} we prove the measure-theoretic theorems, and in Section \ref{SEC TOP THMS} we prove the topological theorems and Proposition \ref{INTRO BOREL}.

\section{The Key Lemma} \label{SEC KEY LEM}

Before stating the key lemma, we quickly review some terminology. Let $G$ and $H$ be countable groups, and let $G$ act on a set $X$. A \emph{$G$-cocycle} is a map $\alpha : G \times X \rightarrow H$ satisfying the \emph{cocycle identity}:
$$\forall g_1, g_2 \in G \ \forall x \in X \quad \alpha(g_2 g_1, x) = \alpha(g_2, g_1 \cdot x) \cdot \alpha(g_1, x).$$
Note that if $\alpha$ is independent of the $X$ coordinate then it must be a group homomorphism from $G$ to $H$. If $H$ acts on a set $Y$ and $\alpha : G \times X \rightarrow H$ is a $G$-cocycle, then we can form the \emph{skew product} action of $G$ on $Y \times X$, denoted $G \acts Y \rtimes_\alpha X$, defined by
$$g \cdot (y, x) = (\alpha(g, x) \cdot y, \ g \cdot x).$$
Lastly, recall that for an abelian group $K$, a map $\phi : K \rightarrow K$ is \emph{affine} if there is a group automorphism $\sigma$ of $K$ and $t \in K$ such that $\phi(k) = \sigma(k) \cdot t$ for all $k \in K$. In particular, affine maps are bijective. We let $\Aff(K)$ denote the collection of Borel affine maps on a compact metrizable abelian group $K$.

\begin{klem} \label{KLEM FULL}
Let $G$ be a countably infinite group, and let $\F_2$ be the rank two free group. Let $G$ and $\F_2$ both act on the standard Borel space $Z$, not necessarily freely. Assume that every $\F_2$-orbit is contained in a $G$-orbit. Further assume that there is a Borel $\F_2$-cocycle $\alpha : \F_2 \times Z \rightarrow G$ satisfying $\alpha(f, z) \cdot z = f \cdot z$ for all $f \in \F_2$ and $z \in Z$, and such that the map $f \in \F_2 \mapsto \alpha(f, z) \in G$ is injective for each $z \in Z$. Then there is a compact metrizable abelian group $K$, a $G$-cocycle $\beta: G \times (2^\N)^G \times Z \rightarrow \Aff(K)$, and a $G$-equivariant Borel isomorphism
$$\phi : 2^G \times Z \ \cong \ K \rtimes_\beta \Big( (2^\N)^G \times Z \Big).$$
Moreover, if $G$ and $\F_2$ act continuously on $Z$ and $\alpha$ is continuous, then the induced projection $\pi : 2^G \times Z \rightarrow (2^\N)^G$ is continuous.
\end{klem}

It is important to note that we do not require $G$ and $\F_2$ to act freely on $Z$. Indeed, when $\F_2$ is a subgroup of $G$ we will find it advantageous to let $Z = \{z\}$ be a singleton (with $G$ and $\F_2$ acting trivially) and let $\alpha : \F_2 \times \{z\} \rightarrow G$ correspond to an injective homomorphism. The requirement that $f \mapsto \alpha(f, z)$ be injective for each $z \in Z$ is needed in order to compensate for the fact that $G$ and $\F_2$ might not act freely on $Z$.

For the remainder of this section we let $\F = \F_2$ denote the rank two free group. We first obtain a specialized version of the key lemma when $G = \F$.

\begin{lem} \label{FAKE KLEM}
There is a compact metrizable abelian group $K_0$, a $\F$-cocycle $\beta_0 : \F \times (2^\N)^\F \rightarrow \Aff(K_0)$, and an $\F$-equivariant Borel isomorphism
$$\phi_0 : 2^\F \ \cong \ K_0 \rtimes_{\beta_0} (2^\N)^\F.$$
Moreover, the induced $\F$-equivariant projection $\pi_0 : 2^\F \rightarrow (2^\N)^\F$ is continuous.
\end{lem}

\begin{proof}
Let $a$ and $b$ freely generate $\F$. Let $\psi_0 : 2^\F \rightarrow (2 \times 2)^\F$ be the Ornstein--Weiss factor map \cite{OW87} defined by
$$\psi_0(x)(f) = \bigg( x(f) + x(f a), \ x(f) + x(f b) \bigg) \mod 2.$$
Notice that the value $\psi_0(x)(f)$, $f \in \F$, reveals for each edge $e \in \{(f, fa), (f, fb)\}$ in the Cayley graph of $\F$ whether $x$ labels the endpoints of $e$ with the same value or with different values. Since the Cayley graph of $\F$ is a tree, it easily follows from this observation that $\psi_0$ is an $\F$-equivariant $2$-to-$1$ surjection onto $(2 \times 2)^\F$. Now we follow \'{A}d\'{a}m Tim\'{a}r's method \cite[Prop. 2.1]{Ba05} to obtain an $\F$-equivariant map from $2^\F$ onto $(2^\N)^\F$. Specifically, after noting that $(2 \times 2)^\F = 2^F \times 2^\F$, we see in Figure \ref{FIG2} a sequence of maps which stabilize in the limit to a map $\psi : 2^\F \rightarrow (2^\N)^\F$. The map $\psi$ is continuous since $\psi_0$ is continuous. It is not hard to see that the image of $\psi$ is dense. Thus by continuity and the compactness of $2^{\F}$ we deduce that $\psi$ is surjective. Furthermore, $\psi$ is $\F$-equivariant since $\psi_0$ is $\F$-equivariant.

The spaces $2^\F$ and $(2^\N)^\F$ are compact metrizable abelian groups under the operation of coordinate-wise addition modulo $2$, and $\psi$ is a continuous group homomorphism. Therefore the kernel $K_0 \lhd 2^\F$ of $\psi$ is a compact metrizable abelian group. Since $\psi^{-1}(y)$ is a $K_0$-coset, it is compact for each $y \in (2^\N)^\F$. Therefore there is a one-sided Borel inverse $\sigma : (2^\N)^\F \rightarrow 2^\F$ such that $\psi(\sigma(y)) = y$ for all $y \in (2^\N)^\F$ \cite[Theorem 18.18]{K95}. Define a Borel bijection
$$\phi_0 : K_0 \times (2^\N)^\F \rightarrow 2^\F$$
by $\phi_0(k, y) = k + \sigma(y)$.

Define $\beta_0: \F \times (2^\N)^\F \rightarrow \Aff(K_0)$ by
$$\beta_0(f, y)(k) = f \cdot k + f \cdot \sigma(y) - \sigma(f \cdot y).$$
Note that $K_0$ is invariant under the action of $\F$ since $\psi$ is $\F$-equivariant. Additionally, $\F$ acts on $2^\F$ by automorphisms and thus acts on $K_0$ by automorphisms. Therefore for each $f \in \F$ and $y \in (2^\N)^\F$, $\beta_0(f, y) \in \Aff(K_0)$ as required. This cocycle gives us a skew product action $\F \acts K_0 \rtimes_{\beta_0} (2^\N)^\F$. Now to complete the proof, observe that for $k \in K_0$ and $y \in (2^\N)^\F$
\begin{align*}
\phi_0 \Big( \beta_0(f, y)(k), f \cdot y \Big) & = \beta_0(f, y)(k) + \sigma(f \cdot y) = f \cdot k + f \cdot \sigma(y)\\
 & = f \cdot \phi_0(k, y).
\end{align*}
Thus $\phi_0$ is $\F$-equivariant.
\end{proof}

\begin{figure}
\setlength{\unitlength}{0.5cm}
\centering
\begin{picture}(5.5,9)
\put(0, 7.5){\makebox{$\displaystyle{2^\F}$}}
\put(0.3, 7.2){\vector(0,-1){1}}
\put(0.4, 6.4){\makebox{$\displaystyle{^{\psi_0}}$}}
\put(0, 5.4){\makebox{$\displaystyle{2^\F \times 2^\F}$}}
\put(0.3, 5.1){\vector(0,-1){1}}
\put(0.4, 4.3){\makebox{$\displaystyle{^\id}$}}
\put(1.9, 5.1){\vector(0,-1){1}}
\put(2, 4.3){\makebox{$\displaystyle{^{\psi_0}}$}}
\put(0, 3.3){\makebox{$\displaystyle{2^\F \times 2^\F \times 2^\F}$}}
\put(0.3, 3){\vector(0,-1){1}}
\put(0.4, 2.2){\makebox{$\displaystyle{^{\id}}$}}
\put(1.9, 3){\vector(0,-1){1}}
\put(2, 2.2){\makebox{$\displaystyle{^{\id}}$}}
\put(3.5, 3){\vector(0,-1){1}}
\put(3.6, 2.2){\makebox{$\displaystyle{^{\psi_0}}$}}
\put(0, 1.2){\makebox{$\displaystyle{2^\F \times 2^\F \times 2^\F \times 2^\F}$}}
\put(2.7, 0.7){\makebox{.}}
\put(2.7,0.35){\makebox{.}}
\put(2.7,0){\makebox{.}}
\end{picture}
\caption{A sequence of maps which stabilize in the limit. Here $\id$ denotes the identity map. \label{FIG2}}
\end{figure}

For the remainder of this section we assume the requirements of the Key Lemma \ref{KLEM FULL} are satisfied. That is, fix a countable group $G$, fix a standard Borel space $Z$, and fix Borel actions of $G$ and $\F$ on $Z$ (not necessarily free actions). Assume that every $\F$-orbit is contained in a $G$-orbit and assume that there is a Borel $\F$-cocycle $\alpha : \F \times Z \rightarrow G$ with $\alpha(f, z) \cdot z = f \cdot z$ for all $f \in \F$ and $z \in Z$ and such that the map $f \mapsto \alpha(f, z)$ is injective for every $z \in Z$. We will deduce the key lemma from the above lemma by using a coinduction argument, in a fashion similar to work of Bowen \cite{Bo12}, Epstein \cite{E07}, and Stepin \cite{St75}. Basically, the idea is to use the $\F$-action on $Z$ in order to view $2^G \times Z$ as $2^{\F \times \N} \times Z$ (here the $\N$ appears since each $G$-orbit on $Z$ may break into a countably infinite number of $\F$-orbits) and then apply the above lemma to obtain a bijection with $K_0^\N \times (2^\N)^{\F \times \N} \times Z$ which, finally, may be viewed as $K_0^\N \times (2^\N)^G \times Z$. 

Recall that given a Borel action $G \acts X$, a set $T \subseteq X$ is a \emph{transversal} for the $G$-orbits if $T$ meets every $G$-orbit in precisely one point.

\begin{lem} \label{LEM SMOOTH}
Fix an injective homomorphism $\iota : \F \rightarrow \F$ whose image has infinite index in $\F$. Then the operation
$$f \cdot (g, z) = \Big(g \cdot \alpha(\iota(f), g^{-1} \cdot z)^{-1}, \ \ z \Big)$$
defines a free Borel action of $\F$ on $G \times Z$ which commutes with the diagonal action of $G$. Furthermore, there is a Borel injection $c : \N \times Z \rightarrow G \times Z$ whose image is a transversal of the $\F$-orbits and which satisfies $c(n, z) \in G \times \{z\}$ and $c(0, z) = (1_G, z)$ for all $z \in Z$ and $n \in \N$.
\end{lem}

\begin{proof}
We temporarily consider the operation $f * (g, z) = (g \cdot \alpha(f, g^{-1} \cdot z)^{-1}, z)$, where $\iota$ is intentionally excluded (compare with the statement of the lemma). Fix $u, f \in \F$, $g \in G$, and $z \in Z$. Notice that $\alpha(f, g^{-1} \cdot z) \cdot g^{-1} \cdot z = f \cdot (g^{-1} \cdot z)$. So by the cocycle identity
\begin{align*}
& \ \alpha \Big(u, \ \alpha(f, g^{-1} \cdot z) \cdot g^{-1} \cdot z \Big)\\
= & \ \alpha \Big(u, \ f \cdot (g^{-1} \cdot z) \Big)\\
= & \ \alpha \Big(u, \ f \cdot (g^{-1} \cdot z) \Big) \cdot \alpha \Big(f, g^{-1} \cdot z \Big) \cdot \alpha \Big(f, g^{-1} \cdot z \Big)^{-1}\\
= & \ \alpha(u f, g^{-1} \cdot z) \cdot \alpha \Big(f, g^{-1} \cdot z \Big)^{-1}.
\end{align*}
Therefore, setting $h = g \cdot \alpha(f, g^{-1} \cdot z)^{-1} \in G$, we have
$$\alpha \Big( u, h^{-1} \cdot z \Big) \cdot h^{-1} = \alpha \Big(u, \ \alpha(f, g^{-1} \cdot z) \cdot g^{-1} \cdot z \Big) \cdot \alpha \Big(f, g^{-1} \cdot z \Big) \cdot g^{-1} = \alpha \Big( u f, g^{-1} \cdot z \Big) \cdot g^{-1}.$$
It follows that
$$u * f * (g, z) = u * (h, z) = (h \cdot \alpha(u, h^{-1} \cdot z)^{-1}, \ \ z) = (g \cdot \alpha(u f, g^{-1} \cdot z)^{-1}, \ \ z) = (u f) * (g, z).$$
So $*$ is an action of $\F$.

Since the map $f \mapsto \alpha(f, z) \in G$ is injective for all $z \in Z$, the action $*$ of $\F$ on $G \times Z$ is free. The operation $\cdot$ defined in the statement of this lemma is related to $*$ by $f \cdot (g, z) = \iota(f) * (g, z)$. So $\cdot$ is a free action of $\F$ on $G \times Z$. Since $*$ is a free action and the image of $\iota$ has infinite index in $\F$, each $*$ invariant set $G \times \{z\}$ decomposes into a countably infinite number of $\cdot$ $\F$-orbits. Fix an enumeration $1_G = g_0, g_1, \ldots$ of $G$. For $z \in Z$ set $c(0, z) = (1_G, z)$ and inductively define
$$c(n, z) = (g_k, z) \qquad \text{where } k \text{ is least with } (g_k, z) \neq f \cdot c(i, z) \text{ for all } f \in \F \text{ and } 0 \leq i < n.$$
Note that $c(n, z)$ is defined since the set $G \times \{z\}$ contains infinitely many $\F$-orbits. Then $c$ is a Borel injection and its image is a transversal of the $\F$-orbits. Finally, this action commutes with the diagonal action of $G$ since for $g, h \in G$ and $f \in \F$ we have $\alpha(\iota(f), g^{-1} \cdot z) = \alpha(\iota(f), g^{-1} h^{-1} \cdot (h \cdot z))$.
\end{proof}

\begin{lem} \label{LEM CO1}
Define $\delta : G \times Z \rightarrow \Sym(\N)$ by
$$\delta(g, z)(n) = k \Longleftrightarrow g \cdot c(n, z) \in \F \cdot c(k, g \cdot z).$$
Then $\delta$ is a $G$-cocycle.
\end{lem}

\begin{proof}
Fix $z \in Z$, $g, h \in G$, and $n \in \N$. Set $k = \delta(g, z)(n)$ and $m = \delta(h, g \cdot z)(k)$. By definition, there are $u, v \in \F$ with
$$g \cdot c(n, z) = u \cdot c(k, g \cdot z) \quad \text{and} \quad h \cdot c(k, g \cdot z) = v \cdot c(m, h g \cdot z).$$
By the previous lemma, the actions of $\F$ and $G$ on $G \times Z$ commute. Therefore
$$h g \cdot c(n, z) = h \cdot (u \cdot c(k, g \cdot z)) = u \cdot (h \cdot c(k, g \cdot z)) = u v \cdot c(m, h g \cdot z).$$
So by the definition of $\delta$ we have
$$\delta(h g, z)(n) = m = \delta(h, g \cdot z)(k) = \delta(h, g \cdot z) \circ \delta(g, z) (n).$$
Since $n \in \N$ was arbitrary we conclude that $\delta(h g, z) = \delta(h, g \cdot z) \circ \delta(g, z)$.
\end{proof}

\begin{lem} \label{LEM CO2}
Define $\gamma : G \times \N \times Z \rightarrow \F$ by
$$\gamma(g, n, z) = u \Longleftrightarrow g \cdot c(n, z) = u^{-1} \cdot c(k, g \cdot z) \quad \text{where} \quad k = \delta(g, z)(n).$$
Then $\gamma$ is a $G$-cocycle.
\end{lem}

\begin{proof}
Fix $z \in Z$, $n \in \N$, and $g, h \in G$. Set $k = \delta(g, z)(n)$ and $m = \delta(h, g \cdot z)(k) = \delta(h g, z)(n)$. Also set $u = \gamma(g, n, z)$ and $v = \gamma(h, k, g \cdot z)$. By the definitions of $\delta$ and $\gamma$ we have
$$g \cdot c(n, z) = u^{-1} \cdot c(k, g \cdot z) \quad \text{and} \quad h \cdot c(k, g \cdot z) = v^{-1} \cdot c(m, h g \cdot z).$$
Therefore, using the commutativity of the actions of $G$ and $\F$ on $G \times Z$, we obtain
$$h g \cdot c(n, z) = h \cdot (u^{-1} \cdot c(k, g \cdot z)) = u^{-1} \cdot (h \cdot c(k, g \cdot z)) = u^{-1} v^{-1} \cdot c(m, h g \cdot z).$$
Now it follows from the definition of $\gamma$ that
$$\gamma(h g, n, z) = v u = \gamma(h, k, g \cdot z) \cdot \gamma(g, n, z).$$
The action of $G$ on $\N \rtimes_\delta Z$ is such that $g \cdot (n, z) = (k, g \cdot z)$. Thus the above expression is the cocycle identity.
\end{proof}

\begin{cor} \label{COR BIJECT}
The map $\psi : \F \times \N \times Z \rightarrow G \times Z$ defined by
$$\psi(f, n, z) = f^{-1} \cdot c(n, z) \in G \times \{z\}$$
is a Borel bijection which conjugates the diagonal action of $G$ on $G \times Z$ to the double skew product action of $G$ on $\F \rtimes_\gamma (\N \rtimes_\delta Z)$.
\end{cor}

\begin{proof}
The map $\psi$ is bijective since the image of $c$ is a transversal for the $\F$-orbits and $\F$ acts freely on $G \times Z$. Also, $\psi$ is Borel since $c$ is Borel and the action of $\F$ is Borel. Finally, fix $g \in G$, $f \in \F$, $n \in \N$, and $z \in Z$. Set $k = \delta(g, z)(n)$ and $u = \gamma(g, n, z)$. From the definitions of $\delta$ and $\gamma$ we have
$$g \cdot c(n, z) = u^{-1} \cdot c(k, g \cdot z).$$
Note that the action $G \acts \F \rtimes_\gamma (\N \rtimes_\delta Z)$ is such that $g \cdot (f, n, z) = (u f, k, g \cdot z)$. Since the actions of $G$ and $\F$ on $G \times Z$ commute we obtain
\begin{align*}
g \cdot \psi(f, n, z) & = g \cdot f^{-1} \cdot c(n, z) = f^{-1} \cdot g \cdot c(n, z)\\
 & = f^{-1} u^{-1} \cdot c(k, g \cdot z) = \psi(u f, k, g \cdot z)\\
 & = \psi(g \cdot (f, n, z)).
\end{align*}
We conclude that $\psi$ is $G$-equivariant.
\end{proof}

\begin{cor} \label{COR BIJECT2}
For $z \in Z$ define $\psi_z : \F \times \N \rightarrow G$ implicitly by the relation
$$\psi(f, n, z) = (\psi_z(f, n), z).$$
Then each $\psi_z$ is a bijection, $\psi_z(f, 0) = \alpha(\iota(f)^{-1}, z)^{-1}$, and for $g \in G$, $f \in \F$, and $n \in \N$
$$g^{-1} \cdot \psi_{g \cdot z}(f, n) = \psi_z \Big( \gamma(g^{-1}, n, g \cdot z) \cdot f, \ \ \delta(g^{-1}, g \cdot z)(n) \Big).$$
\end{cor}

\begin{proof}
Each $\psi_z$ must be bijective since $\psi$ is bijective. Since $c(0, z) = (1_G, z)$ we have
$$\psi(f, 0, z) = f^{-1} \cdot c(0, z) = f^{-1} \cdot (1_G, z) = (\alpha(\iota(f)^{-1}, z)^{-1}, z).$$
So $\psi_z(f, 0) = \alpha(\iota(f)^{-1}, z)^{-1}$ as required. Finally, for $g \in G$, $f \in \F$, and $n \in \N$, we use the fact that $\psi$ is $G$-equivariant to obtain
\begin{align*}
\Big( g^{-1} \cdot \psi_{g \cdot z}(f, n), \ \ z \Big) & = g^{-1} \cdot \Big( \psi_{g \cdot z}(f, n), \ \ g \cdot z \Big)\\
 & = g^{-1} \cdot \psi(f, n, g \cdot z)\\
 & = \psi \Big( \gamma(g^{-1}, n, g \cdot z), \ \ \delta(g^{-1}, g \cdot z)(n), \ \ z \Big)\\
 & = \Big( \psi_z \Big( \gamma(g^{-1}, n , g \cdot z), \ \ \delta(g^{-1}, g \cdot z)(n) \Big), \ \ z \Big). \qedhere
\end{align*}
\end{proof}

Let $\F \acts X$ be a Borel action of $\F$. Then we obtain an action of $G$ on $X^\N \times Z$ given by
\begin{equation} \label{eqn:act}
g \cdot (x, z) = (x', g \cdot z) \quad \text{where} \quad x'(n) = \gamma(g^{-1}, n, g \cdot z)^{-1} \cdot x \Big(\delta(g^{-1}, g \cdot z)(n) \Big).
\end{equation}
When $A$ is a compact metrizable space we will implicitly identify $(A^\F)^\N$ with $A^{\F \times \N}$. In the special case $X = A^\F$, the action of $G$ on $X^\N = A^{\F \times \N}$ can be described by the formula
\begin{equation} \label{eqn:act2}
g \cdot (x, z) = (x', g \cdot z) \quad \text{where} \quad x'(n)(f) = x \Big(\delta(g^{-1}, g \cdot z)(n) \Big) \Big(\gamma(g^{-1}, n, g \cdot z) \cdot f \Big).
\end{equation}
The utility of these actions is manifest in the following two lemmas.

\begin{lem} \label{LEM DUAL}
Let $A$ be a compact metrizable space. Then the map $\psi^* : A^G \times Z \rightarrow A^{\F \times \N} \times Z$ given by
$$\psi^*(x, z) = (\psi^*_z(x), z) \quad \text{where} \quad \psi^*_z(x)(n)(f) = x(\psi_z(f, n))$$
is a $G$-equivariant Borel bijection.
\end{lem}

\begin{proof}
Clearly $\psi^*$ is a Borel bijection since $\psi_z : \F \times \N \rightarrow G$ is a bijection for each $z \in Z$. Suppose that $\psi^*(x, z) = (y, z)$ and $\psi^*(g \cdot x, g \cdot z) = (y', g \cdot z)$. Then, using only the definition of $\psi^*$ and Corollary \ref{COR BIJECT2}, we obtain
\begin{align*}
y'(n)(f) & = (g \cdot x) \Big( \psi_{g \cdot z}(f, n) \Big)\\
 & = x \Big( g^{-1} \cdot \psi_{g \cdot z}(f, n) \Big)\\
 & = x \Big( \psi_z \Big( \gamma(g^{-1}, n, g \cdot z) \cdot f, \ \ \delta(g^{-1}, g \cdot z)(n) \Big) \Big)\\
 & = y \Big( \delta(g^{-1}, g \cdot z)(n) \Big) \Big( \gamma(g^{-1}, n, g \cdot z) \cdot f \Big).
\end{align*}
By (\ref{eqn:act2}) we see that $(y', g \cdot z) = g \cdot (y, z)$. We conclude that $\psi^*$ is $G$-equivariant.
\end{proof}

\begin{lem} \label{LEM EQUIV}
Let $\F \acts X$ and $\F \acts Y$ be Borel actions of $\F$, and suppose that $\theta : X \rightarrow Y$ is $\F$-equivariant. Then $\theta^\N \times \id : X^\N \times Z \rightarrow Y^\N \times Z$ is $G$-equivariant.
\end{lem}

\begin{proof}
Fix $x \in X^\N$ and $z \in Z$. Then $\theta^\N \times \id (x, z) = (y, z)$ where $y(n) = \theta(x(n))$. Set $(x', g \cdot z) = g \cdot (x, z)$. Then $\theta^\N \times \id (x', g \cdot z) = (y', g \cdot z)$ where $y'(n) = \theta(x'(n))$. The definition of the $G$ action on $X^\N \times Z$ says that
$$x'(n) = \gamma(g^{-1}, n, g \cdot z)^{-1} \cdot x \Big( \delta(g^{-1}, g \cdot z)(n) \Big).$$
Since $\theta$ is $\F$-equivariant we obtain
$$y'(n) = \theta(x'(n)) = \gamma(g^{-1}, n, g \cdot z)^{-1} \cdot \theta \Big( x \Big( \delta(g^{-1}, g \cdot z)(n) \Big) \Big)$$
$$= \gamma(g^{-1}, n, g \cdot z)^{-1} \cdot y \Big(\delta(g^{-1}, g \cdot z)(n) \Big).$$
By (\ref{eqn:act}) we see that $(y', g \cdot z) = g \cdot (y, z)$. We conclude that $\theta^\N \times \id$ is $G$-equivariant.
\end{proof}

We now prove the key lemma. We state here a redacted version of the lemma with the assumptions omitted.

\begin{klem*}
There is a compact metrizable abelian group $K$, a $G$-cocycle $\beta: G \times (2^\N)^G \times Z \rightarrow \Aff(K)$, and a $G$-equivariant Borel isomorphism
$$\phi: 2^G \times Z \ \cong \ K \rtimes_\beta \Big( (2^\N)^G \times Z \Big).$$
Moreover, if $G$ and $\F$ act continuously on $Z$ and $\alpha$ is continuous, then the induced projection $\pi : 2^G \times Z \rightarrow (2^\N)^G$ is continuous.
\end{klem*}

\begin{proof}
Let $\phi_0$, $\beta_0$, $K_0$, $\delta$, $\gamma$, $\psi$, and $\psi^*$ be as in the above lemmas. We have a sequence of Borel bijections taking $2^G \times Z$ to $K_0^\N \times (2^\N)^G \times Z$, as depicted in Figure \ref{FIG1}. We must define a $G$-cocycle $\beta: G \times (2^\N)^G \times Z \rightarrow \Aff(K_0^\N)$ and verify that the final bijection in Figure \ref{FIG1}, $\id \times \psi^*$, is $G$-equivariant. All other bijections appearing in Figure \ref{FIG1} are $G$-equivariant by the lemmas above.

\begin{figure}
\setlength{\unitlength}{0.5cm}
\centering
\begin{picture}(20.75,8)
\put(0, 6){\makebox{$\displaystyle{2^G \times Z}$}}
\put(2.625, 6.25){\vector(1,0){3}}
\put(4, 6.25){\makebox{$\displaystyle{^{\psi^*}}$}}
\put(6, 6){\makebox{$\displaystyle{\left(\prod_{n \in \N} 2^\F \right) \times Z}$}}
\put(8.25, 5){\vector(0,-1){2}}
\put(8.31, 3.625){\makebox{$\displaystyle{^{\phi_0^\N \times \id}}$}}
\put(4.6825, 1.75){\makebox{$\displaystyle{K_0^\N \times \left(\prod_{n \in \N} (2^\N)^\F \right) \times Z}$}}
\put(15.3125, 2){\vector(-1,0){3}}
\put(13.375, 2){\makebox{$\displaystyle{^{\id \times \psi^*}}$}}
\put(15.6875, 1.75){\makebox{$\displaystyle{K_0^\N \times (2^\N)^G \times Z}$}}
\end{picture}
\caption{A sequence of Borel bijections from $2^G \times Z$ onto $K_0^\N \times (2^\N)^G \times Z$. \label{FIG1}}
\end{figure}

We have that $\psi^*: (2^\N)^G \times Z \rightarrow (2^\N)^{\F \times \N} \times Z$ is $G$-equivariant by Lemma \ref{LEM DUAL}. So $\id \times \psi^* : K_0^\N \times (2^\N)^G \times Z \rightarrow K_0^\N \times (2^\N)^{\F \times \N} \times Z$ will be $G$-equivariant as long as we correctly define $\beta : G \times (2^\N)^G \times Z \rightarrow \Aff(K_0^\N)$ in terms of $\beta_0$. Observe that for $z \in Z$ and $x \in (2^\N)^G$ we have $\psi_z^*(x) \in (2^\N)^{\F \times \N}$. Define $\beta : G \times (2^\N)^G \times Z \rightarrow \Aff(K_0^\N)$ by
\begin{align*}
\beta( & g, x, z)(w)(n)\\
& = \beta_0 \Bigg( \gamma(g^{-1}, n, g \cdot z)^{-1}, \ \ \psi^*_z(x) \Big( \delta(g^{-1}, g \cdot z)(n) \Big) \Bigg) \Bigg( w \Big( \delta(g^{-1}, g \cdot z)(n) \Big) \Bigg).
\end{align*}
The function $\beta(g, x, z) : K_0^\N \rightarrow K_0^\N$ permutes the $\N$-coordinates and then applies affine maps coordinate-wise. From this observation it follows that $\beta(g, x, z) \in \Aff(K_0^\N)$.

Fix $g \in G$, $z \in Z$, $x \in (2^\N)^G$, and $w \in K_0^\N$. Let $w'$ be such that $g \cdot (w, x, z) = (w', g \cdot x, g \cdot z)$. In other words, $w' = \beta(g, x, z)(w)$, or equivalently for all $n \in \N$
$$w'(n) = \beta_0 \Bigg( \gamma(g^{-1}, n, g \cdot z)^{-1}, \ \ \psi^*_z(x) \Big( \delta(g^{-1}, g \cdot z)(n) \Big) \Bigg) \Bigg( w \Big( \delta(g^{-1}, g \cdot z)(n) \Big) \Bigg).$$
Let $y, y' \in (2^\N)^{\F \times \N}$ be such that $(w, y, z) = \id \times \psi^*(w, x, z)$ and $(w', y', g \cdot z) = \id \times \psi^*(w', g \cdot x, g \cdot z)$. Note that $(y', g \cdot z) = g \cdot (y, z)$ since $\psi^*$ is $G$-equivariant. In order to show that $\id \times \psi^*$ is $G$-equivariant, it suffices to check that $g \cdot (w, y, z) = (w', y', g \cdot z)$. The action of $G$ on $K_0^\N \times (2^\N)^{\F \times \N} \times Z$ is given by (\ref{eqn:act}), where $X = K_0 \times (2^\N)^\F$. Thus $g \cdot (w, y, z) = (t, y', g \cdot z)$ where for $n \in \N$
\begin{align*}
(t(n), y'(n)) & = (t, y')(n) = \gamma(g^{-1}, n, g \cdot z)^{-1} \cdot (w, y) \Big(\delta(g^{-1}, g \cdot z)(n) \Big)\\
 & = \gamma(g^{-1}, n, g \cdot z)^{-1} \cdot \Big( w(\delta(g^{-1}, g \cdot z)(n)), \ \ y(\delta(g^{-1}, g \cdot z)(n)) \Big).
\end{align*}
By considering the $\F$-action $\F \acts K_0 \rtimes_{\beta_0} (2^\N)^\F$ we see that for $n \in \N$
$$t(n) = \beta_0 \Big( \gamma(g^{-1}, n, g \cdot z)^{-1}, \ \ y (\delta(g^{-1}, g \cdot z)(n)) \Big) \Big( w(\delta(g^{-1}, g \cdot z)(n)) \Big).$$
By definition $(y, z) = \psi^*(x, z)$ and hence $y = \psi_z^*(x)$. So $t = \beta(g, x, z)(w) = w'$. We conclude that $\id \times \psi^*$ is $G$-equivariant.

Now suppose that $G$ and $\F$ act continuously on $Z$ and that $\alpha$ is continuous. Let $\pi : 2^G \times Z \rightarrow (2^\N)^G$ be the induced factor map, and let $\pi_0$ be the function referred to in Lemma \ref{FAKE KLEM}. Consider $x \in 2^G$ and $z \in Z$. Let $w \in K_0^\N$ and $x' \in (2^\N)^G$ be such that $\phi(x, z) = (w, x', z)$, where $\phi : 2^G \times Z \rightarrow K_0^\N \times (2^\N)^G \times Z$ is the function constructed via Figure \ref{FIG1}. Then $x' = \pi(x, z)$. Set $y = \psi_z^*(x)$ and $y' = \psi_z^*(x')$. Then $y \in 2^{\F \times \N}$, $y' \in (2^\N)^{\F \times \N}$ and from Figure \ref{FIG1} we see that $\pi_0^\N(y) = y'$. If $g \in G$ and $\psi_z(f, n) = g$ then the element $\pi(x, z)(g) \in 2^\N$ satisfies
\begin{align*}
\pi(x, z)(g) & = x'(g) = \psi_z^*(x')(n)(f) = y'(n)(f)\\
 & = \pi_0^\N(y)(n)(f) = \pi_0 \Big( y(n) \Big)(f) = \pi_0 \Big( \psi_z^*(x)(n) \Big)(f).
\end{align*}
Therefore for $m \in \N$ we have
$$\pi(x, z)(g)(m) = \pi_0 \Big( \psi_z^*(x)(n) \Big)(f)(m) \quad \text{where} \quad \psi_z(f, n) = g.$$
Since $G$ acts by homeomorphisms on $2^G \times Z$ and $(2^\N)^G$ and since $\pi$ is $G$-equivariant, it will suffice to check that for every $m \in \N$ the map $(x, z) \mapsto \pi(x, z)(1_G)(m)$ is continuous. Since $\psi_z(1_\F, 0) = 1_G$, we have
$$\pi(x, z)(1_G)(m) = \pi_0 \Big( \psi_z^*(x)(0) \Big)(1_\F)(m).$$
Evaluation at $(1_\F)(m)$ is continuous and $\pi_0$ is continuous by Lemma \ref{FAKE KLEM}, so it is enough to check that $(x, z) \mapsto \psi_z^*(x)(0) \in 2^\F$ is continuous. For this we need only check that
$$(x, z) \mapsto \psi_z^*(x)(0)(f)$$
is continuous for each fixed $f \in \F$. By Lemma \ref{LEM DUAL} and Corollary \ref{COR BIJECT2} we have
$$\psi_z^*(x)(0)(f) = x(\psi_z(f, 0)) = x(\alpha(\iota(f)^{-1}, z)^{-1}).$$
Fix $(x, z) \in 2^G \times Z$ and $f \in \F$. Since $\alpha$ is continuous there is an open neighborhood $V$ of $z$ such that $\alpha(\iota(f)^{-1}, z')^{-1} = \alpha(\iota(f)^{-1}, z)^{-1}$ for all $z' \in V$. Set $g = \alpha(\iota(f)^{-1}, z)^{-1}$ and let $U$ be the open set $\{x' \in 2^G \: x'(g) = x(g)\}$. Then for all $(x', z') \in U \times V$ we have $\psi_{z'}^*(x')(0)(f) = \psi_z^*(x)(0)(f)$. Thus $\pi$ is continuous as claimed.
\end{proof}

Before closing this section, we present an application of these methods.

\begin{thm}[Bowen] \label{THM BOWEN}
Let $G$ be a countable non-amenable group. Let $\fact$ be the infimum of $\sH(M, \mu)$ over all probability spaces $(M, \mu)$ with the property that $G \acts (M^G, \mu^G)$ factors onto all other Bernoulli shifts over $G$. Let $\von$ be the infimum of $\sH(M, \mu)$ over all probability spaces $(M, \mu)$ with the property that there is an essentially free action of $\F$ on $(M^G, \mu^G)$ such that $\mu^G$-almost-every $\F$-orbit is contained in a $G$-orbit. Then $\fact = \von < \infty$.
\end{thm}

\begin{proof}
We have that $\fact < \infty$ by Ball's theorem \cite{Ba05} and $\von < \infty$ by the Gaboriau--Lyons theorem \cite{GL09}. Suppose that $G \acts (M^G, \mu^G)$ factors onto all other Bernoulli shifts over $G$. By the Gaboriau--Lyons theorem \cite{GL09}, there is an essentially free action of $\F$ on $([0,1]^G, \lambda^G)$, where $\lambda$ is Lebesgue measure, such that $\lambda^G$-almost-every $\F$-orbit is contained in a $G$-orbit. Define $\alpha : \F \times [0,1]^G \rightarrow G$ by
$$\alpha(f, x) = g \Longleftrightarrow f \cdot x = g \cdot x.$$
Then $\alpha$ is defined $\lambda^G$-almost-everywhere since $G$ acts on $[0,1]^G$ essentially freely and $\lambda^G$-almost-every $\F$-orbit is contained in a $G$-orbit. Let $G \acts (M^G, \mu^G)$ factor onto $G \acts ([0,1]^G, \lambda^G)$ via $\phi$. Then we let $\F$ act on $(M^G, \mu^G)$ by the rule
$$f \cdot x = \alpha(f, \phi(x)) \cdot x.$$
This is well-defined $\mu^G$-almost-everywhere. Furthermore, this action is essentially free and $\mu^G$-almost-every $\F$-orbit is contained in a $G$-orbit. It follows that $\von \leq \fact$. 

Now suppose that there is an essentially free action of $\F$ on $(M^G, \mu^G)$ such that $\mu^G$-almost-every $\F$-orbit is contained in a $G$-orbit. Let $Z \subseteq M^G$ be a $G$-invariant Borel set of $\mu^G$-full-measure such that both $G$ and $\F$ act freely on $Z$ and every $\F$-orbit on $Z$ is contained in a $G$-orbit. Let $\alpha$ be the $\F$-cocycle relating the $\F$ action to the $G$ action. Then we automatically have that $\alpha$ is Borel and $f \mapsto \alpha(f, z)$ is injective for each $z \in Z$.

Fix $\epsilon > 0$ and let $(K, \kappa)$ be a probability space with $0 < \sH(K, \kappa) < \epsilon$. Also fix an arbitrary Bernoulli shift $(N^G, \nu^G)$ over $G$. By Bowen's theorem \cite{Bo11}, there is a $\F$-equivariant factor map $\phi : (K^\F, \kappa^\F) \rightarrow (N^\F, \nu^\F)$. Since anything not in the domain of $\phi$ can be mapped to a fixed point of $N^\F$, we may assume without loss of generality that $\phi$ is $\F$-equivariant and has domain $K^\F$. By Lemma \ref{LEM DUAL} we have $G$-equivariant Borel bijections
$$\psi^* : K^G \times Z \rightarrow K^{\F \times \N} \times Z \quad \text{and} \quad \psi^* : N^G \times Z \rightarrow N^{\F \times \N} \times Z.$$
Recall that $\psi^*$ is defined by
$$\psi^*(x, z) = (\psi_z^*(x), z) \quad \text{where} \quad \psi_z^*(x)(n)(f) = x(\psi_z(f, n)).$$
Since $\psi_z: \F \times \N \rightarrow G$ is a bijection for each $z \in Z$, $\psi^*$ is simply permuting coordinates for each $z \in Z$. Since the coordinate values in $K^G$ and $N^G$ are identically distributed, we see that
$$\psi^*(\kappa^G \times \mu^G) = \kappa^{\F \times \N} \times \mu^G \quad \text{and} \quad \psi^*(\nu^G \times \mu^G) = \nu^{\F \times \N} \times \mu^G.$$
By Lemma \ref{LEM EQUIV} the map
$$\phi^\N \times \id : K^{\F \times \N} \times Z \rightarrow N^{\F \times \N} \times Z$$
is $G$-equivariant. As $\phi(\kappa^\F) = \nu^\F$, it follows that
$$(\phi^\N \times \id)(\kappa^{\F \times \N} \times \mu^G) = \nu^{\F \times \N} \times \mu^G.$$
Putting these maps together, we obtain a $G$-equivariant Borel map from $K^G \times Z$ into $N^G \times Z$ which pushes $\kappa^G \times \mu^G$ forward to $\nu^G \times \mu^G$. Since $K^G \times Z$ is $\kappa^G \times \mu^G$-conull in $K^G \times M^G$, we deduce that $(K^G \times M^G, \kappa^G \times \mu^G)$ factors onto $(N^G \times M^G, \nu^G \times \mu^G)$, which of course factors onto $(N^G, \nu^G)$. Thus $(K^G \times M^G, \kappa^G \times \mu^G)$ factors onto all other Bernoulli shifts over $G$ and
$$\sH(K \times M, \kappa \times \mu) = \sH(K, \kappa) + \sH(M, \mu) < \sH(M, \mu) + \epsilon.$$
Since $\epsilon > 0$ was arbitrary we deduce that $\fact \leq \von$. 
\end{proof}

\section{Measure-theoretic theorems} \label{SEC MEAS THMS}

In this section we prove that for various actions $G \acts (X, \mu)$ there is a measure $\nu$ on some $n^G$ so that $G \acts (n^G, \nu)$ factors onto $G \acts (X, \mu)$. The following two simple lemmas will allow us to apply the Key Lemma \ref{KLEM FULL}.

\begin{lem} \label{LEM MEASREP}
Let $G$ be a countable group, and let $G \acts X$ be a Borel action of $G$. Then there is a $G$-equivariant Borel injection $\phi : X \rightarrow (2^\N)^G$.
\end{lem}

\begin{proof}
By assumption $X$ is a standard Borel space. Since all uncountable standard Borel spaces are Borel isomorphic \cite[Theorem 15.6]{K95}, there is certainly an injective Borel map $\theta : X \rightarrow 2^\N$. Now define $\phi : X \rightarrow (2^\N)^G$ by $\phi(x)(g) = \theta(g^{-1} \cdot x)$. This map is $G$-equivariant since
$$\phi(g \cdot x)(h) = \theta(h^{-1} g \cdot x) = \phi(x)(g^{-1} h) = [g \cdot \phi(x)] (h).$$
The map $\phi$ is easily seen to be Borel and injective.
\end{proof}

\begin{lem}
Let $K$ be a compact metrizable group, and let $\kappa$ be the Haar probability measure on $K$. Then $\kappa$ is $\Aff(K)$-invariant.
\end{lem}

\begin{proof}
Fix $\theta \in \Aff(K)$. Let $\sigma$ be an automorphism of $K$ and let $t \in K$ be such that $\theta(k) = \sigma(k) \cdot t$ for all $k \in K$.
Note that for $k, a \in K$ we have
$$\theta^{-1}(k a) = \sigma^{-1}(k a t^{-1}) = \sigma^{-1}(k) \cdot \sigma^{-1}(a t^{-1}) = \sigma^{-1}(k) \cdot \theta^{-1}(a).$$
So for a Borel set $A \subseteq K$ and $k \in K$ we have
$$\theta(\kappa)(k \cdot A) = \kappa( \theta^{-1}(k \cdot A)) = \kappa( \sigma^{-1}(k) \cdot \theta^{-1}(A)) = \kappa(\theta^{-1}(A)) = \theta(\kappa)(A).$$
So $\theta(\kappa)$ is translation-invariant and hence $\theta(\kappa) = \kappa$ by the uniqueness of Haar measure.
\end{proof}

In order to obtain better estimates on $n$, we note the following. Recall that two {\pmp} actions $G \acts (X, \mu)$ and $G \acts (Y, \nu)$ are \emph{measurably conjugate} if there exists a $\mu$-almost-everywhere defined $G$-equivariant injective Borel map $\phi : X \rightarrow Y$ which pushes $\mu$ forward to $\nu$.

\begin{lem} \label{LEM GL}
Let $G$ be a countable non-amenable group. If $G$ is not finitely generated or has a non-amenable subgroup of infinite index, then for every Bernoulli shift $(M^G, \mu^G)$ over $G$ there is an ergodic essentially free action $\F_2 \acts (M^G, \mu^G)$ with $\mu^G$-almost-every $\F_2$-orbit contained in a $G$-orbit. In particular, $\fact = \von = 0$.
\end{lem}

\begin{proof}
Since every non-amenable group contains a finitely generated non-amenable subgroup, under either assumption we have that $G$ contains a non-amenable subgroup $H$ of infinite index. Consider a Bernoulli shift $(M^G, \mu^G)$. Then the restricted action $H \acts (M^G, \mu^G)$ is measurably conjugate to $H \acts ((M^{G/H})^H, (\mu^{G/H})^H)$. Since $H$ has infinite index in $G$, $(M^{G / H}, \mu^{G / H})$ is a non-atomic probability space and hence isomorphic to $([0, 1], \lambda)$ where $\lambda$ is Lebesgue measure. So the action $H \acts ((M^{G / H})^H, (\mu^{G / H})^H)$ is measurably conjugate to $H \acts ([0,1]^H, \lambda^H)$. Now by the Gaboriau--Lyons theorem \cite{GL09} there is an ergodic essentially free {\pmp} action $\F_2 \acts ([0,1]^H, \lambda^H)$ with $\lambda^H$-almost-every $\F_2$-orbit contained in an $H$-orbit. Pulling back this action of $\F_2$ through the measure conjugacies, we obtain the desired action of $\F_2$ on $(M^G, \mu^G)$. For the final conclusion, apply Theorem \ref{THM BOWEN}.
\end{proof}

\begin{thm} \label{THM NA MEAS}
For every countable non-amenable group $G$ there exists a finite integer $n$ with the following property: If $G \acts (X, \mu)$ is any {\pmp} action then there is a $G$-invariant Borel probability measure $\nu$ on $n^G$ such that $G \acts (n^G, \nu)$ factors onto $G \acts (X, \mu)$. If $G$ is not finitely generated, or $G$ has a non-amenable subgroup of infinite index, or $G$ has a subgroup with cost greater than one, then $n$ can be chosen to be $4$. If $G$ contains $\F_2$ as a subgroup then $n$ can be chosen to be $2$.
\end{thm}

\begin{proof}
Fix a non-amenable group $G$. Suppose there is $m \in \N$ and a $G$-invariant Borel probability measure $\zeta$ on $m^G$ with the property that there is an action of $\F_2$ on $m^G$ with $\zeta$-almost-every $\F_2$-orbit contained in a $G$-orbit. Further suppose that there is a Borel $\F_2$-cocycle $\alpha : \F_2 \times m^G \rightarrow G$ such that for $\zeta$-almost every $z \in m^G$, the map $f \mapsto \alpha(f, z)$ is injective and $\alpha(f, z) \cdot z = f \cdot z$ for all $f \in \F_2$. Let $Z \subseteq m^G$ be a $\zeta$-conull $G$-invariant Borel set such that for every $z \in Z$, $f \mapsto \alpha(f, z)$ is injective and $\alpha(f, z) \cdot z = f \cdot z$. Set $n = 2 m$.

Consider a {\pmp} action $G \acts (X, \mu)$. By Lemma \ref{LEM MEASREP}, we may assume that $X = (2^\N)^G$. Let $\phi$, $K$ and $\beta : G \times (2^\N)^G \times Z \rightarrow \Aff(K)$ be as described in the Key Lemma \ref{KLEM FULL}. Let $\kappa$ be the Haar probability measure on $K$. Then the action of $G$ on
$$(K \rtimes_\beta (2^\N)^G \times Z, \ \kappa \times \mu \times \zeta)$$
is a {\pmp} action of $G$ which clearly factors onto $((2^\N)^G, \mu)$. Since $\phi : K \rtimes_\beta (2^\N)^G \times Z \rightarrow 2^G \times Z$ is a $G$-equivariant Borel bijection, we have that $(K \rtimes_\beta (2^\N)^G \times Z, \kappa \times \mu \times \zeta)$ is measurably and equivariantly conjugate to $(2^G \times Z, \nu)$, where $\nu$ is the push-forward of $\kappa \times \mu \times \zeta$. Since $2^G \times Z \subseteq 2^G \times m^G = n^G$, we can view $\nu$ as a measure on $n^G$. Then $G \acts (n^G, \nu)$ factors onto $G \acts ((2^\N)^G, \mu)$.

If $G$ contains $\F_2$ as a subgroup, then we can set $m = 1$ so that $1^G = \{1\}$ is a singleton, and define $\alpha : \F_2 \times 1^G \rightarrow G$ by $\alpha(f, 1) = \theta(f)$ where $\theta: \F_2 \rightarrow G$ is an injective homomorphism. In general we can consider a Bernoulli shift $(m^G, u_m^G)$, where $u_m$ is the uniform probability measure on $\{1, 2, \ldots, m\}$, having the property that there is an essentially free action of $\F_2$ on $m^G$ such that $u_m^G$-almost-every $\F_2$-orbit is contained in a $G$-orbit. In this case we can take $\alpha$ to be the $\F_2$-cocycle relating the two actions, and we will automatically have that $f \mapsto \alpha(f, z)$ is injective for $u_m^G$-almost-every $z \in m^G$. The Gaboriau--Lyons theorem \cite{GL09} states that there always exists such an $m$, and furthermore that one can take $m = 2$ if $G$ has a subgroup with cost greater than one. Similarly by Lemma \ref{LEM GL} one can take $m = 2$ if $G$ is not finitely generated or has a non-amenable subgroup of infinite index.
\end{proof}

\begin{thm} \label{THM MEAS}
Let $G$ be a countably infinite group. For every Borel action $G \acts X$ and every quasi-invariant Borel measure $\mu$ on $X$, there is a quasi-invariant Borel measure $\nu$ on $4^G$ such that $G \acts (4^G, \nu)$ factors onto $G \acts (X, \mu)$.
\end{thm}

\begin{proof}
Let $z \in 2^G$ be defined by $z(1_G) = 1$ and $z(g) = 0$ for all $1_G \neq g \in G$. Set $Z = G \cdot z$. Fix any quasi-invariant probability measure $\zeta$ on the countable set $Z$ (this can be done by assigning a strictly positive real number to each element of $Z$ in a manner that the cumulative sum is $1$). Fix any bijection $\F_2 \rightarrow G$ and use this bijection to let $\F_2$ act transitively on $Z$. Let $\alpha$ be the $\F_2$-cocycle relating the $\F_2$-action to the $G$-action. Then $\alpha$ is Borel, and $f \mapsto \alpha(f, z')$ is injective for each $z' \in Z$. Now consider a Borel action $G \acts X$ and let $\mu$ be a quasi-invariant Borel measure on $X$. By Lemma \ref{LEM MEASREP} we may assume that $X = (2^\N)^G$. Let $\phi$, $K$, and $\beta$ be as in the Key Lemma \ref{KLEM FULL}. Let $\kappa$ be the Haar probability measure on $K$. Then $(K \rtimes_\beta (2^\N)^G \times Z, \kappa \times \mu \times \zeta)$ equivariantly factors onto $((2^\N)^G, \mu)$ since $\kappa \times \zeta$ is a probability measure. Furthermore, $\kappa \times \mu \times \zeta$ is quasi-invariant. Now use the $G$-equivariant Borel bijection $\phi : K \times (2^\N)^G \times Z \rightarrow 2^G \times Z$ to obtain the desired Borel measure $\nu$ on $2^G \times Z \subseteq 4^G$.
\end{proof}

\section{Topological theorems} \label{SEC TOP THMS}

We now prove topological analogues of the theorems from the previous section. The role of the Gaboriau--Lyons theorem \cite{GL09} in the previous section, which is considered to be a measurable solution to the von Neumann conjecture, is replaced in this section by a theorem of Whyte \cite{Wh99} which is considered to be a geometric solution to the von Neumann conjecture.

Recall that a subset $Y$ of a topological space $X$ is $G_\delta$ if it is the intersection of a countable family of open sets. A basic fact which we will use in this section is that for a Polish space $X$, a subspace $Y \subseteq X$ is Polish if and only if $Y$ is a $G_\delta$ subset of $X$ \cite[Theorem 3.11]{K95}.

\begin{lem} \label{LEM TOPREP}
Let $G$ be a countable group and let $G \acts X$ be a continuous action of $G$ on a Polish space $X$. Then there is a $G$-invariant Polish subspace $Y \subseteq (2^\N)^G$ such that $G \acts Y$ continuously factors onto $G \acts X$. Furthermore, if $X$ is compact then $Y$ can be chosen to be compact.
\end{lem}

\begin{proof}
We first claim that there is a Polish subspace $Y_0 \subseteq 2^\N$ and a continuous surjection $\theta: Y_0 \rightarrow X$, with $Y_0$ compact if $X$ is compact. Since $2^\N$ continuously surjects onto all compact metric spaces \cite[Theorem 4.18]{K95}, if $X$ is compact we may take $Y_0 = 2^\N$ and let $\theta : 2^\N \rightarrow X$ be a continuous surjection. In the general case, let $Y_0$ be the set of $y \in 2^\N$ such that $y(n) = 1$ for infinitely many $n$. Then $Y_0$ is $G_\delta$ in $2^\N$ and thus Polish. Also, $Y_0$ is homeomorphic to $\N^\N$ via the map $\psi : Y_0 \rightarrow \N^\N$
\begin{align*}
\psi(y)(n) = m - k \quad \quad \text{where } m, k & \in \N \text{ are least with}\\
 & \sum_{i = 0}^{m-1} y(i) = n+1 \text{ and } \sum_{i=0}^{k-1} y(i) = n.
\end{align*}
Since every Polish space is the continuous image of $\N^\N$ \cite[Theorem 7.9]{K95}, there is a continuous surjection $\theta : Y_0 \rightarrow X$.

Now consider the Polish space $Y_0^G \subseteq (2^\N)^G$. Note that it is compact if $X$ is compact. Define $\phi : Y_0^G \rightarrow X$ by $\phi(y) = \theta(y(1_G))$. Now we let $Y$ be the set of points where this map is $G$-equivariant, specifically
$$Y = \{y \in Y_0^G \: \forall g \in G \quad g \cdot \phi(y) = \phi(g \cdot y)\}.$$
Note that $\phi$ maps $Y$ onto $X$ since $\theta$ is surjective. Furthermore, since $G$ acts continuously and $\phi$ is continuous, $Y$ is a closed subset of $Y_0^G \subseteq (2^\N)^G$. Hence $Y$ is Polish and is compact if $X$ is compact, and $\phi : Y \rightarrow X$ is a $G$-equivariant surjection.
\end{proof}

\begin{thm}
For every countable non-amenable group $G$ there exists a finite integer $n$ with the following property: If $G \acts X$ is a continuous action on a compact metrizable space, then there is a $G$-invariant compact set $Y \subseteq n^G$ such that $G \acts Y$ continuously factors onto $G \acts X$. If $G$ contains $\F_2$ as a subgroup then $n$ can be chosen to be $2$.
\end{thm}

\begin{proof}
Fix a non-amenable group $G$. Suppose that there is $m \in \N$ and a $G$-invariant compact set $Z \subseteq m^G$ with the property that there is a continuous action of $\F_2$ on $Z$ with every $\F_2$-orbit contained in a $G$-orbit. Further suppose that there is a continuous $\F_2$-cocycle $\alpha : \F_2 \times Z \rightarrow G$ such that for every $z \in Z$, $f \mapsto \alpha(f, z)$ is injective and $\alpha(f, z) \cdot z = f \cdot z$ for all $f \in \F_2$. Set $n = 2m$.

Consider a continuous action $G \acts X$ where $X$ is a compact metrizable space. By Lemma \ref{LEM TOPREP} we may assume that $X \subseteq (2^\N)^G$. By the Key Lemma \ref{KLEM FULL} $G \acts 2^G \times Z$ continuously factors onto $G \acts (2^\N)^G$ via a map $\pi$. Set $Y = \pi^{-1}(X)$. Since $X$ is compact it is a closed subset of $(2^\N)^G$, and thus $Y$ is closed since $\pi$ is continuous. Then $Y$ is a closed subset of the compact space $2^G \times Z$. Thus $Y$ is compact and $Y \subseteq 2^G \times Z \subseteq n^G$. Furthermore, $Y$ is clearly $G$-invariant and $G \acts Y$ continuously factors onto $G \acts X$ via $\pi$.

If $G$ contains $\F_2$ as a subgroup then we can set $m = 1$ so that $1^G = \{1\}$ is a singleton, and define $\alpha : \F_2 \times 1^G \rightarrow G$ by $\alpha(f, 1) = \theta(f)$ where $\theta: \F_2 \rightarrow G$ is an injective homomorphism. In the general case, let $H \leq G$ be a finitely generated non-amenable subgroup of $G$, and let $a$ and $b$ freely generate $\F_2$. By Whyte's solution to the geometric von Neumann conjecture \cite{Wh99}, there is a finite set $S \subseteq H$ and a free action $*$ of $\F_2$ on $H$ such that $a * h, b * h \in S h$ for all $h \in H$ (in fact the action can be chosen to be transitive \cite{Se11}). Consider the (symbolic) Bernoulli shift $(S \times S)^G$. We define two maps $T_a, T_b : (S \times S)^G \rightarrow (S \times S)^G$ by
\begin{align*}
T_a(x) = s \cdot x \Longleftrightarrow x(1_G)(1) = s\\
T_b(x) = s \cdot x \Longleftrightarrow x(1_G)(2) = s.
\end{align*}
Let $Z_0$ consist of those points $x \in (S \times S)^G$ such that $|T_a^{-1}(x)| = |T_b^{-1}(x)| = 1$. The complement of $Z_0$ is open, so $Z_0$ is closed. Set $Z_1 = \bigcap_{g \in G} g \cdot Z_0$ and note that $Z_1$ is closed. We let $\F_2$ act on $Z_1$ by setting $a \cdot x = T_a(x)$ and $b \cdot x = T_b(x)$ for $x \in Z_1$. Clearly every $\F_2$-orbit on $Z_1$ is contained in a $G$-orbit, $\F_2$ acts continuously, and the $\F_2$-cocycle $\alpha : \F_2 \times Z_1 \rightarrow G$ relating the $\F_2$-action to the $G$-action is continuous. Now define
$$Z_2 = \{x \in Z_1 \: \forall f \in \F_2 \quad f \neq 1_{\F_2} \Longrightarrow f \cdot x \neq x\}.$$
It is readily verified that the complement of $Z_2$ is open, and thus $Z_2$ is closed. Now set $Z = \bigcap_{g \in G} g \cdot Z_2$. Then $Z$ is a $G$-invariant closed subset of $(S \times S)^G$, hence $G$-invariant and compact. Furthermore, $f \mapsto \alpha(f, z)$ is injective for each $z \in Z$ since $\F_2$ acts freely on $Z$. Most importantly, by Whyte's theorem $Z$ is non-empty. Thus for any non-amenable group there are $m$, $Z$, and $\alpha$ with the required properties.
\end{proof}

\begin{thm}
Let $G$ be a countably infinite group. For every continuous action $G \acts X$ on a Polish space $X$, there is a $G$-invariant Polish subspace $Y \subseteq 4^G$ such that $G \acts Y$ continuously factors onto $G \acts X$.
\end{thm}

\begin{proof}
Let $z \in 2^G$ be defined by $z(1_G) = 1$ and $z(g) = 0$ for all $1_G \neq g \in G$. Set $Z = G \cdot z$. Note that the relative topology on $Z$ from $2^G$ is discrete and thus $Z$ is Polish. Fix any bijection $\F_2 \rightarrow G$ and use this bijection to let $\F_2$ act transitively on $Z$. Let $\alpha$ be the $\F_2$-cocycle relating the $\F_2$-action to the $G$-action. Since the relative topology on $Z$ is discrete, we see that $\F_2$ acts continuously and that $\alpha$ is continuous. Now consider a continuous action $G \acts X$ with $X$ Polish. By Lemma \ref{LEM TOPREP} we may assume $X \subseteq (2^\N)^G$. By the Key Lemma \ref{KLEM FULL} $G \acts 2^G \times Z$ continuously factors onto $G \acts (2^\N)^G$ via a map $\pi$. Set $Y = \pi^{-1}(X)$. Since $X$ is Polish, it is a $G_\delta$-subset of $(2^\N)^G$. As $\pi$ is continuous, $Y$ must be a $G_\delta$ subset of the Polish space $2^G \times Z$. Thus $Y$ is Polish and $Y \subseteq 2^G \times Z \subseteq 4^G$. Finally, we observe that $Y$ is $G$-invariant and $\pi : Y \rightarrow X$ is a $G$-equivariant continuous surjection. 
\end{proof}

Lastly, we prove a simple Borel dynamics analogue of these results.

\begin{prop}
Let $G$ be a countable group, and let $G \acts Y$ be a Borel action. Suppose there is an uncountable invariant set $Y_0 \subseteq Y$ such that the action $G \acts Y_0$ is free. Then for every Borel action $G \acts X$ there is a $G$-invariant Borel set $Z \subseteq Y$ such that $G \acts Z$ Borel factors onto $G \acts X$. Moreover, if $X$ has a fixed point then $G \acts Y$ Borel factors onto $G \acts X$.
\end{prop}

\begin{proof}
By letting $Y_0$ be the set of $y \in Y$ with $g \cdot y \neq y$ for all $1_G \neq g \in G$, we may assume that $Y_0$ is Borel. Since $Y_0 \subseteq Y$ Borel, there is some Polish topology on $Y_0$ which is compatible with the Borel structure $Y_0$ inherits from $Y$ \cite[Cor. 13.4]{K95}. Fix such a topology. By \cite[Theorem 13.11 and Lemma 13.3]{K95}, we may assume that $G$ acts continuously on $Y_0$. By the Cantor--Bendixon theorem \cite[Theorem 6.4]{K95}, there is a $G$-invariant non-empty closed set $Y' \subseteq Y_0$ such that $Y'$ is perfect, meaning it has no isolated points. In particular, every non-empty open set in $Y'$ is uncountable. Let $R \subseteq Y' \times Y'$ consist of those pairs $(y_1, y_2)$ with $y_1 \not\in G \cdot y_2$. Then $R$ is a dense $G_\delta$ subset of $Y' \times Y'$ and thus by the Kuratowski--Mycielski theorem \cite[Theorem 19.1]{K95} there is an uncountable compact set $K \subseteq Y'$ such that $(k_1, k_2) \in R$ for all $k_1 \neq k_2 \in K$. So no two elements of $K$ lie in the same $G$-orbit (alternatively, one could obtain a similar conclusion by applying Silver's theorem \cite{Si80}).

Set $Z = G \cdot K$. Then $Z$ is $G$-invariant and Borel, and $K$ is a transversal for the $G$-orbits on $Z$. Now consider a Borel action $G \acts X$. Since $K$ is an uncountable standard Borel space, there is a Borel surjection $\theta : K \rightarrow X$ \cite[Theorem 15.6]{K95}. Now extend $\theta$ to $Z$ by setting $\theta(g \cdot k) = g \cdot \theta(k)$ for $g \in G$ and $k \in K$. This is well defined since $G$ acts freely on $Z$ and each $G$-orbit in $Z$ contains precisely one point in $K$. Clearly $\theta$ is $G$-equivariant and maps $Z$ onto $X$. Furthermore, $\theta$ is Borel since it is Borel on each class of the countable Borel partition $\{g \cdot K \: g \in G\}$ of $Z$. Finally, if $X$ has a fixed point $x_0$, meaning $g \cdot x_0 = x_0$ for all $g \in G$, then we can extend $\theta$ to all of $Y$ be setting $\theta(y) = x_0$ for $y \in Y \setminus Z$.
\end{proof}

\end{document}